\documentclass[12pt,psamsfonts]{amsart}
\usepackage{amsthm}
\usepackage{amssymb}
\usepackage{amsmath}
\usepackage{graphicx}
\usepackage{amscd}
\usepackage{amsfonts}
\usepackage{amsbsy}
\usepackage[T1]{fontenc}
\usepackage[english]{babel}
\usepackage{xcolor}
\usepackage{a4wide}
\usepackage{hyperref}

\usepackage{epsfig}
\usepackage{amssymb}

\newtheorem{theorem}{Theorem}%[section]
\newtheorem{corollary}[theorem]{Corollary}

\newtheorem{proposition}[theorem]{Proposition}

\newtheorem{lemma}[theorem]{Lemma}

\newcommand{\Z}{\mathbb Z}
\newcommand{\Q}{\mathbb Q}
\newcommand{\N}{\mathbb N}

\newcommand{\C}{\mathbb C}
\newcommand{\HH}{\mathbb H}

\newcommand{\zf}{\zeta_f}
\newcommand{\ra}{\rightarrow}

\newcommand{\mperL}{\mbox{MPer}_{L}(f)}

\newcommand{\SSS}{{\mathbb S}}

\author[V.F. Sirvent]
{V\'{\i}ctor F. Sirvent}

\address{ 
Departamento de Matem\'aticas, Universidad Cat\'olica del Norte, Antofagasta, Chile}
\email{victor.sirvent@ucn.cl}

\title[Periodic structure of maps on the product of spheres]{On the periodic structure of $C^1$ self-maps on the product of spheres of different dimensions} 

\date{\today}



\begin{document}

\begin{abstract}

In the present article we study the periodic structure of some well-known classes of $C^1$ self-maps on the product of spheres of different dimensions: 
transversal maps, Morse-Smale diffeomorphisms and maps with all its periodic points hyperbolic. 
Our approach is via the Lefschetz fixed point theory.
We give a complete characterization of the minimal set of Lefschetz periods for Morse-Smale diffeomorphisms on these spaces.
We also consider $C^1$ maps with all its periodic points hyperbolic and  we give  conditions for these maps to have infinitely many periodic points.
We describe the period set of the transversal maps on these spaces.
Finally we applied these results to describe the periodic structure of similar classes of maps on compact, connected, simply connected  and simple Lie groups.
\end{abstract}

\maketitle

\section{Introduction}\label{s:introduction}

The study of the periodic structure of a map (i.e. periodic points, periodic orbits
and their corresponding periods) is an important subject in the theory of dynamical systems. 
Indeed, in many situations the knowledge of the periodic structure helps to understand other (global) properties like chaotic behavior, entropy, etc.
There are different approaches and techniques for this study, one of the most well-known and used is the Lefschetz  fixed point theory, see the monograph~\cite{JM}, for a comprehensive exposition  of the subject.

\smallskip

In the present paper we consider self-maps on the product of spheres of different dimensions, i.e. $\SSS^{n_1}\times\cdots\times\SSS^{n_l}$ with $1\leq n_1<\cdots<n_l$. 
Our approach is via the Lefschetz  fixed point theory.
The articles~\cite{LL-S:2018,Si:2020}  studied self-maps on products of spheres of different  dimension in the context of period-free maps, i.e. the maps do not have periodic points.  
In the present article we compute  the Lefschetz zeta function for these maps and use it to describe the dynamics of different and well-known classes of differentiable maps on these spaces: Morse-Smale diffeomorphisms, maps with all its periodic points hyperbolic and transversal maps.

\smallskip
The structure of the article is as follows:
In section~\ref{s:lzf} we compute  (and give an explicit and  closed formula for) the Lefschetz zeta function  for continuous self maps on these spaces (Theorem~\ref{thm:lzf}). We use this result in the following sections.
In section~\ref{s:mperl} we consider the minimal sets of Lefschetz periods for Morse-Smale diffeomorphisms on these spaces and give a complete characterization of this set in Theorem~\ref{thm:mperl}. 
This theorem generalizes previous known results in~\cite{CC-LL}, where criteria are given for self-maps on the product of two spheres of different dimensions.
In section~\ref{s:hyperbolic} we consider  $C^1$ maps with all its periodic points hyperbolic and in Theorem~\ref{thm:pp} we give sufficient conditions for these maps to have infinitely many periodic points.
In section~\ref{s:transversal} we consider transversal maps on these spaces. In Lemma~\ref{lemma:transversal} and Theorem~\ref{thm:main-transversal} we give an extensive characterization of the elements of the period set of transversal maps on the spaces considered here.
Theorem~\ref{thm:main-transversal} generalizes previous results presented  in~\cite{LL-S:2019} for the case of transversal maps on a sum-free product of spheres (i.e. when all the homology groups are trivial or one-dimensional). 
The proofs in the present article are simpler than the ones in~\cite{LL-S:2019}.
We conclude the article with section~\ref{s:remarks} where some applications, remarks and open questions are presented.
One of the main applications of the results of this article is the study of the periodic structure of $C^1$ maps on some Lie groups. This is done via Hopf's Theorem~(Theorem~\ref{thm:hopf}) which states these spaces have the same homology as a product of spheres of different odd dimensions. In this way the Corollaries~\ref{coro:ms-lie}, \ref{coro:pp-lie} and \ref{coro:transversal-lie} follows from the main theorems of the article.

\section{The Lefschetz numbers and the Lefschetz zeta function}\label{s:lzf}

Let $X$ be a compact topological manifold of dimension $N$ and $f:X\to X$ a continuous map, and 
$f_{*k}:H_k(X,\Q)\to H_k(X,\Q)$ the induced maps on the $k$th-rational homology groups of $X$, with $0\leq k\leq N$.
The \emph{Lefschetz number} $L(f)$
is defined as
\begin{equation*}\label{number}
L(f):= \sum\limits_{k=0}^{N}(-1)^{k} \text{trace}(f_{*k}).
\end{equation*}

The \emph{Lefschetz zeta function} of $f$ is the exponential generating function of the Lefschetz numbers of the iterates of $f$, i.e.
$$
\zf(t):=\exp\left(\sum_{m=1}^{\infty}  \frac{L(f^m)}{m}t^m\right).
$$

It  can be expressed formally as an infinite Euler product~(\emph{cf.}~\cite{dold}), of the form:
\begin{equation}\label{eqn:lzf-euler}
\zf(t)=\prod_{m\geq 1}(1-t^m)^{-\dfrac{\ell(f^m)}{m}},
\end{equation}
where $\ell(f^m)$ is the \emph{Lefschetz number of period $m$}, 
defined as
\begin{equation}\label{eqn:l-chica}
\ell(f^m):=\sum_{r|m}\mu(r) L(f^{m/r})=\sum_{r|m}\mu({m/r}) L(f^r),
\end{equation}
where $\mu$ is the classical \emph{M\"obius function}, i.e.
 $$
\mu(m):=\left\{\begin{array}{ll}
               1 & \mbox{ if } m=1,\\
               0 & \mbox{ if } k^2|m \mbox{ for some } k\in\N, \\
              (-1)^r & \mbox{ if } m=p_1\cdots p_r  \mbox{ has distinct  primes factors.}
              \end{array}
       \right.
$$
By the M\"obius inversion formula~(\emph{cf.}~\cite{hardy-wright}  Theorem 266):
$$
L(f^m)=\sum_{r|m} \ell(f^r).
$$
The Lefschetz zeta function is a rational function since it can be written as
%There is an alternative way to compute the Lefschetz zeta function:
\begin{equation}\label{eqn:zeta-product}
\zf(t)=\prod_{k=0}^N\det(Id_{*k}-t f_{*k})^{(-1)^{k+1}},
\end{equation}
where $N=\dim M$,
%$n_k=\dim H_k(M,\Q)$, 
$Id:=Id_{*k}$ is the
identity map on $H_k(M,\Q)$, and by convention $\det(Id_{*k}-t
f_{*k})=1$ if $\dim H_k(M,\Q)=0$, for more details see~\cite{franks2}.

Throughout the article we consider the space a product of spheres of different dimensions, i.e.
$X=\SSS^{n_1}\times\cdots\times\SSS^{n_l}$ with $1\leq n_1<\cdots<n_l$.
The homology groups of $X$ are obtained by using  the
K\"unneth Theorem~(\emph{cf.}~\cite{hatcher}), we get
$$
H_k(X,\Q)=\underbrace{\Q\oplus\cdots\oplus\Q}_{b_k}.
$$
where $b_k$ is the number of ways that $k$ can be written by summing up subsets of $(n_{i_1},\cdots,n_{i_l})$.
The numbers $b_k$ are called the \emph{Betti numbers} of the space $X$ and they are the coefficients  of the Poincar\'e polynomial of the space.

Let $f:X\to X$ be a continuous map;
if the homology groups $H_{n_j}(X,\Q)$, with $1\leq j\leq l$, are one dimensional then 
the corresponding induced maps on homology are
$f_{*n_j}=(a_{n_j})$ 
where the numbers $a_{n_j}$ are integers. 

\begin{proposition}[\cite{Si:2020}]\label{prop:ev}
%Let $f:X\to X$ be a continuous map and 
Let $H_k(X,\Q)$  be non-trivial.
If $\lambda$ is an eigenvalue of $f_{*k}$ then
$\lambda=a_{n_{i_1}}\cdots a_{n_{i_s}}$, where
$n_{i_1}+\cdots+n_{i_s}=k$.
%where $a_{n_j}=f_{*n_j}$ with $1\leq j\leq l$.
\end{proposition}

We say that the numbers $a_{n_1},\ldots,a_{n_l}$ are the 
\emph{basic eigenvalues of $f_{*}$}, if $f_{*n_i}=(a_{n_i})$ when $H_{n_i}(X,\Q)$ is one dimensional, for $1\leq i\leq l$; or 
$$
f_{*n_i}=\left(\begin{array}{cccc}
a_{n_i} & c(i)_{1,2} &\cdots &c(i)_{1,d+1}\\
0 & a_{n_{j_{1_1}}}\cdots a_{n_{j_{1_r}}} & \cdots &c(i)_{2,d+1}\\
0&0 &\vdots & c(i)_{d,d+1}\\
 0& 0 & \cdots & a_{n_{j_{d_1}}}\cdots a_{n_{j_{d_{r'}}}}
\end{array}
\right),
$$
with $c(i)_{r,t}\in\Z$,
when  the dimension of $H_{n_i}(X,\Q)$ is $d+1$ and
$$n_{j_{1_1}}+\cdots +n_{j_{1_r}}=\cdots=n_{j_{d_1}}+\cdots +n_{j_{d_{r'}}}=n_i,$$ with
$1\leq j_{1_1}<\cdots<j_{i_r}<i$, $\ldots$, $1\leq j_{d_1}<\cdots<j_{d_{r'}}<i$, and $d\geq1$.
In other words the eigenvalues of the induced map on homology are given by the 
basic eigenvalues and some multiplications of the basic eigenvalues,
for more details see~\cite{Si:2020}.
We can say that these numbers are the most basic information of the map, since from them we can obtain many homological invariants of the the map as we shall show in throughout the article. 
The following proposition states that the Lefschetz numbers of $f$ are given by a product, whose factors involves the basic eigenvalues:

\begin{proposition}[\cite{Si:2020}]\label{prop:L-numbers}
Let $X=\SSS^{n_1}\times\cdots\times\SSS^{n_l}$ with $1\leq n_1<\cdots<n_l$, 
and $f:X\to X$ be a continuous map, such that 
$a_{n_1},\ldots, a_{n_l}$, are the basic eigenvalues of $f_{*}$.
Then
\begin{equation}\label{eqn:L-m}
L(f^m)=\left(1+(-1)^{n_1}a_{n_1}^m\right)\cdots \left(1+(-1)^{n_l}a_{n_l}^m\right).
\end{equation}
\end{proposition}

From Proposition~\ref{prop:ev} and the equality~(\ref{eqn:zeta-product}),
the Lefschetz zeta function of a map satisfying the  hypothesis of Proposition~\ref{prop:L-numbers}, can be written as

\begin{equation}\label{eqn:lzf-prod1}
\zf(t)=(1-t)^{-1}\prod_{k=1}^N\left(\prod_{j_1+\cdots+j_r=k} (1-a_{n_{j_1}}\cdots a_{n_{j_r}}t)^{(-1)^{k+1}} \right),
\end{equation}
where $N=n_1+\cdots+n_l$, i.e.  the dimension of the space $X$.
However this formula is not very explicit since we need to find the suitable indexes $j_1,\ldots,j_r$ such that $j_1+\cdots+j_r=k$ (notice that are as many elements as the dimension of $H_k(X,\Q)$, i.e. the $k$-th Betti number) which it is computational expensive, when the Betti numbers are large.

In the following theorem we give an explicit formula of the Lefschetz zeta function in terms of the basic eigenvalues of the map, in order to obtain it we use 
the expression~(\ref{eqn:L-m}) of the Lefschetz numbers.
The formula~(\ref{eqn:lzf-product}) is more computational efficient to obtain than~(\ref{eqn:lzf-prod1}), since it only requires to order the numbers $a_{n_i}$ according to their indices.

\begin{theorem}\label{thm:lzf}
Let $X=\SSS^{n_1}\times\cdots\times\SSS^{n_l}$ with $1\leq n_1<\cdots<n_l$, 
and let $f:X\to X$ be a continuous map, such that 
$a_{n_1},\ldots, a_{n_l}$, are the basic eigenvalues of $f_{*}$.
Then

\begin{equation}\label{eqn:lzf-product}
\zf(t)=\dfrac{1}{1-t}\,\prod_{k=1}^l\left(\prod_{i_{1}<\cdots <i_{k}}(1-a_{n_{i_1}}\cdots a_{n_{i_k}}t)^{(-1)^{n_{i_1}+\cdots +n_{i_k}-1}}\right).
\end{equation}
\end{theorem}

\begin{proof}
The proof follows by induction on $l$:
If $l=1$, the formula~(\ref{eqn:lzf-product}) is the well-known expression of the Lefschetz zeta function of a self-map on a sphere of dimension $n$, for completeness we show the computation here: Let $f_{*n}=(a_n)$, 
\begin{eqnarray*}
\zf(t)&=&\exp\left(\sum_{m=1}^{\infty}  \frac{L(f^m)}{m}t^m\right)=\exp\left(\sum_{m=1}^{\infty}  \frac{(1+(-1)^na_n^m)}{m}t^m\right)\\
  &=& \exp\left(\sum_{m=1}^{\infty} \frac{t^m}{m}\right) \exp\left(\sum_{m=1}^{\infty} \frac{(-1)^n(a_nt)^m)}{m}\right)\\
  &=&\left(1-t\right)^{-1}\left(1-a_nt\right)^{{(-1)}^{n-1}}.
\end{eqnarray*}

Let $f$ be a map, satisfying the hypothesis of the Theorem~\ref{thm:lzf}, so
\begin{eqnarray*}
\zf(t)&=&\exp\left(\sum_{m=1}^{\infty}  \frac{L(f^m)}{m}t^m\right)\\
 &=& \exp\left(\sum_{m=1}^{\infty}(1+(-1)^{n_1}a_{n_1}^m)\cdots (1+(-1)^{n_l}a_{n_l}^m)\dfrac{t^m}{m}\right)\\
 %&=& exp\left(\sum_{m=1}^{\infty}(1+(-1)^{n_1}a_{n_1}^m)\cdots (1+(-1)^{n_{l-1}}a_{n_{l-1}}^m)\right)\dfrac{t^m}{m}
  &=& \exp\left(\sum_{m=1}^{\infty}\left(\prod_{i=1}^{l-1}(1+(-1)^{n_i}a_{n_i}^m)\right)\left( 1+(-1)^{n_l}a_{n_l}^m\right)\dfrac{t^m}{m}\right).
  \end{eqnarray*}
Hence
\begin{equation}
\zf(t)=\exp\left(\sum_{m=1}^{\infty}\left(\prod_{i=1}^{l-1}(1+(-1)^{n_i}a_{n_i}^m)\right)\dfrac{t^m}{m}\right)
  \exp\left((-1)^{n_l}\sum_{m=1}^{\infty}\left(\prod_{i=1}^{l-1}(1+(-1)^{n_i}a_{n_i}^m)\right)\dfrac{(a_{n_l}t)^m}{m}\right).\label{eqn:induction}
\end{equation}

Let $g$ be a continuous self-map on $\SSS^{n_1}\times\cdots\times\SSS^{n_{l-1}}$ with
$a_{n_1},\ldots, a_{n_{l-1}}$ as basic eigenvalues of $g_*$, 
from the definition of the Lefschetz zeta function and Proposition~\ref{prop:L-numbers} it follows
$$
\zeta_g(t)= \exp\left(\sum_{m=1}^{\infty}\left(\prod_{i=1}^{l-1}(1+(-1)^{n_i}a_{n_i}^m)\right)\dfrac{t^m}{m}\right).
$$
So the identity~(\ref{eqn:induction})  can be written as $\zf(t)=\zeta_g(t)\zeta_g(a_{n_l}t)^{(-1)^{n_l}}$.

The inductive hypothesis states
$$
\zeta_g(t)=\dfrac{1}{1-t}\,\prod_{k=1}^{l-1}\left(\prod_{i_{1}<\cdots <i_{k}}(1-a_{n_{i_1}}\cdots a_{n_{i_k}}t)^{(-1)^{n_{i_1}+\cdots +n_{i_k}-1}}\right).
$$
Consider the product 
\begin{multline*}
\zeta_g(t)\zeta_g(a_lt)^{(-1)^{n_l}}=
\dfrac{\,\left(\prod_{k=1}^{l-1}\prod_{i_{1}<\cdots <i_{k}}(1-a_{n_{i_1}}\cdots a_{n_{i_k}}t)^{(-1)^{n_{i_1}+\cdots +n_{i_k}-1}}\right)}{1-t} \times\\
\times\dfrac{\,\left(\prod_{k=1}^{l-1}\prod_{i_{1}<\cdots <i_{k}}(1-a_{n_{i_1}}\cdots a_{n_{i_k}}(a_{n_l} t)^{(-1)^{n_{i_1}+\c1dots +n_{i_k}-1}}\right)^{(-1)^{n_l}}}{(1-a_{n_l}t)^{(-1)^{n_l}}}.
\end{multline*}
Reorganizing the terms in this expression, it yields
\begin{eqnarray*}
\zf(t) &=& \zeta_g(t)\zeta_g(a_{n_l}t)^{(-1)^{n_l}}\\
&=& \dfrac{1}{1-t}\left(\prod_{i=1}^l(1-a_{n_i}t)^{(-1)^{n_i-1}}\right)
\left(\prod_{k=2}^l\left(\prod_{i_{1}<\cdots <i_{k}}(1-a_{n_{i_1}}\cdots a_{n_{i_k}}t)^{(-1)^{n_{i_1}+\cdots +n_{i_k}-1}}\right)\right).
\end{eqnarray*}
Which is the identity~(\ref{eqn:lzf-product}). 
\end{proof}

We say that a map $f$ is  \emph{quasi-unipotent} if the eigenvalues of the $f_{*k}$ are roots of identity for all $k\in\{0,\ldots,N\}$.
These  maps are very important in dynamics since they are associated to the well-known Morse-Smale diffeomorphisms.
The Morse-Smale diffeomorphims are diffeomorphisms such that non-wandering set consists only of a finite number of hyperbolic periodic point,
whose stable and unstable manifolds intersect transversally, see~\cite{shub-sullivan} or~\cite{LL-S:2008} for a precise definition of these maps and the concepts mentioned.
M. Shub~(\cite{shub}) proved that the Morse-Smale diffeomorphisms  are quasi-unipotent.

For quasi-unipotent maps  the expression~(\ref{eqn:lzf-euler}) is a product of finite number of factors and the Lefschetz numbers are bounded.

If $f:\SSS^{n_1}\times\cdots\times\SSS^{n_l}\to \SSS^{n_1}\times\cdots\times\SSS^{n_l}$ with $n_1<\cdots<n_l$,  is  a quasi-unipotent map, then $a_k=\pm 1$; therefore $\det(Id_{*k}-tf_{*k})=(1\pm t)$, whenever $H_k(X,\Q)$ is not trivial (due to Proposition~\ref{prop:ev}).
By~(\ref{eqn:zeta-product}), it follows
 \begin{equation}\label{eqn:lzf-qu}
\zf(t)=(1-t)^\alpha(1+t)^\beta
\end{equation}
where $\alpha$ and $\beta$ are integers. 
The precise values of the exponents $\alpha$ and $\beta$ are computed in  section~\ref{s:remarks}.

\section{Minimal set of Lefschetz periods}\label{s:mperl}

The minimal set of Lefschetz  periods is an important set  for understanding the periodic structure of $C^1$ quasi-unipotent maps, in particular the Morse-Smale diffeomorphisms. 
For a detail description of the concept of the minimal set of Lefschetz periods see for example~\cite{LL-S:2008,LL-S:2013}.
In order to  the article be self-contained we give a brief introduction of this notion and its definition
in the following lines.

Let $X$ be a $C^1$ compact manifold and $f:X\ra X$ a $C^1$ map.
We say that $x$ is a \emph{hyperbolic point of} $f$ if $x$ is a periodic point  of period $p$   of $f$ (for some $p$) and  the eigenvalues of $Df^p(x)$ have modulus different from $1$.
Let $T_xM$ be the tangent space of $M$ at $x$, in this situation
 $T_xM=E_x^u\oplus E_x^s$, where $E_x^u$ the unstable space,
 i.e. the subspace of the tangent space
$T_x M$ generated by the eigenvectors of $Df^p(x)$ of modulus larger
than $1$; and  $E_x^s$ the stable space, defined in a similar manner.
The {\em index}  $u$ of the orbit of $x$ is the dimension of the 
unstable space at $x$.
We define the {\it orientation}
type $\Delta$ of the orbit of $x$ as $+1$ if $Df^p(x):E_x^u\ra E_x^u$
preserves orientation and $-1$ if reverses the orientation. 
The
collection of the triples $(p,u,\Delta)$ belonging to all periodic
orbits of $f$ is called the {\em periodic data} of $f$. 

The following Theorem  is one of the most important results that relates the period data of a differentiable map with the expression of its Lefschetz zeta function.

\begin{theorem}[Franks~\cite{franks1}]\label{thm:franks}
Let $f$ be a $C^1$ a map on a compact manifold without boundary
having finitely many periodic points all of them hyperbolic, and let
$\Sigma$ be the periodic data of $f$. Then the Lefschetz zeta
function $\zf(t)$ of $f$ satisfies
\begin{equation}\label{eqn:zeta-franks}
\zf(t)= \prod_{(p,u,\Delta)\in\Sigma} (1-\Delta t^p)^{(-1)^{u+1}}.
\end{equation}
\end{theorem}

If $\zf(t)\neq 1$ then it can be written as
\begin{equation}\label{eqn:2}
\zf(t)=\prod_{i=1}^{N_{\zeta}}(1+\Delta_i t^{r_i})^{m_i},
\end{equation}
where $\Delta_i=\pm 1$, the $r_i$'s are positive integers, $m_i$'s
are nonzero integers and $N_{\zeta}$ is a positive integer depending
on $f$.

\smallskip

If $\zf(t)\neq 1$ the  {\em minimal set of Lefschetz periods}  of $f$
is defined as
\[
\mperL:= \bigcap  \,\, \{r_1,\ldots, r_{N_{\zeta}}\},
\]
where the intersection is considered over all the possible
expressions (\ref{eqn:2}) of $\zf(t)$. If $\zf(t)= 1$ then we define
$\mperL :=\emptyset$. 
We can say that the minimal set of Lefschetz
periods of $f$ is the intersection of all the sets of periods forced
by the finitely many different representations of $\zf(t)$ as
products of the form $(1\pm t^p)^{\pm1}$.
 Clearly the minimal set of Lefschetz periods is a subset of the set of periods of the map.
 When $\mperL$ is empty we do not get any real information about the 
 period set of $f$.

An important result is that the set $\mperL$ consists only of odd positive integers~({\em cf.}~\cite{LL-S:2008,LL-S:2013}). Note that
if the exponents $r_i$, of ~(\ref{eqn:2}) are odd and pairwise distinct, then the representation~(\ref{eqn:2}) of the Lefschetz zeta function is unique.
%as a product of terms of the form $1\pm t^{r_i}$.

From these facts and the expression~(\ref{eqn:lzf-qu}) follows that the minimal set of Lefschetz periods for a $C^1$ quasi-unipotent map is either the empty set or $\{1\}$. 
When $\mperL=\{1\}$, it means that $f$ has a fixed point. In the case  $\mperL=\emptyset$, we do not get any information about the period set of $f$. 
The following lemma gives a complete characterization of the set for this class of maps.

\begin{lemma}\label{lemma:mperl}
Let $X=\SSS^{n_1}\times\cdots\times\SSS^{n_l}$ with $1\leq n_1<\cdots<n_l$, 
and let $f:X\to X$ be a $C^1$ quasi-unipotent map, such that 
$a_{n_1},\ldots, a_{n_l}$, are the basic eigenvalues of $f_{*}$. 
\begin{enumerate}
\item[(a)] If $(-1)^{n_i}a_{n_i}=1$ for all $1\leq i\leq l$ then $\mperL=\{1\}$.
\item[(b)] If $(-1)^{n_i}a_{n_i}=-1$ for some $1\leq i\leq l$ then $\mperL=\emptyset$.
\end{enumerate}
\end{lemma}
\begin{proof}
By~(\ref{eqn:L-m}), if $(-1)^{n_i}a_{n_i}=1$ for all $1\leq i\leq l$ then $L(f^m)=2^l$ for $m$ odd.
Therefore, when $m$ is odd
$$
\ell(f^m)=\sum_{r|m}\mu(r) L(f^{m/r})=\sum_{r|m}\mu(r) 2^l=2^l \sum_{r|m}\mu(r)=
\left\{\begin{array}{ll} 0 &\text{ if } m>1\\ 2^l &\text{ if } m=1.
\end{array} \right.
$$
Here we used a well-known property of the M\"obius function~(\emph{cf.}~\cite[Theorem 263, p. 235]{hardy-wright}): 
\begin{itemize}
\item if $m>1$ then $\sum_{r|m}\mu(r)=0$,  and 
\item if $m=1$ then $\sum_{r|m}\mu(r)=1$.
\end{itemize}
Since $\mperL$ does not contain even numbers, we conclude that $1$ is the only element of the set $\mperL$.

\smallskip

If   $(-1)^{n_i}a_{n_i}=-1$ for some $1\leq i\leq l$ then $L(f^m)=0$ for all odd positive integers  $m$. Therefore
$\ell(f^m)=0$ for all  $m$ odd. Hence $\mperL$ is the empty set.
\end{proof}

Since the Morse-Smale diffeomorphisms are $C^1$ quasi-unipotent maps, Lemma~\ref{lemma:mperl} yields the following result: 

\begin{theorem}\label{thm:mperl}
Let $X=\SSS^{n_1}\times\cdots\times\SSS^{n_l}$ with $1\leq  n_1<\cdots<n_l$, 
and let $f:X\to X$ be a Morse-Smale diffeomorphism, such that 
$a_{n_1},\ldots, a_{n_l}$, are the basic eigenvalues of $f_{*}$. Then the minimal set of Lefschetz periods is either the empty set or $\{1\}$.
Moreover
\begin{enumerate}
\item If $(-1)^{n_i}a_{n_i}=1$ for all $1\leq i\leq l$ then $\mperL=\{1\}$.
\item If $(-1)^{n_i}a_{n_i}=-1$ for some $1\leq i\leq l$ then $\mperL=\emptyset$.
\end{enumerate}
\end{theorem}

In~\cite{CC-LL} considered the  minimal set of Lefschetz periods for Morse-Smale diffeomorphisms on the product of two spheres, therefore Theorem~\ref{thm:mperl} is a more general version of the results in~\cite{CC-LL}.
The period set of Morse-Smale diffeomorphisms on $\SSS^2$, where studied before in~\cite{G-LL:2008}.
See~\cite{GLM} for other developments in the subject of the minimal sets of Lefschetz periods.

%%%%%%%%%%%%%%%%%%%%%%%%%%%%%%%%%%%%%%%

\section{Maps with finitely many periodic points all of them hyperbolic}\label{s:hyperbolic}

Theorem~\ref{thm:franks} is an  important tool for characterizing  differentiable maps having a finite number of periodic points all of them hyperbolic, and it is done via the Lefschetz zeta function. This has been studied before by different contexts, see for example~\cite{LL-S:2016,G-LL:2021}, and within references.
For a more analytic approach to the subject, see for instance~\cite{morales}.

Given a $C^1$ self-map on the spaces considered  here with all its periodic point hyperbolic, 
in the following theorem
we give sufficient conditions for  the map to have an infinite number of periodic points, in terms of the basic eigenvalues of the induced maps on homology. 
Theorem~4 of~\cite{G-LL:2021} gives other sufficient condition for this type of map to have an infinite number of periodic points.
The criterion given in Theorem~\ref{thm:pp} is easier to check since it is stated in terms of the basic eigenvalues of $f_{*}$, which is the more basic information (from the homology point of view) of the map $f$.

\begin{theorem}\label{thm:pp}
Let $X=\SSS^{n_1}\times\cdots\times\SSS^{n_l}$ with $1\leq n_1<\cdots<n_l$, 
and let $f:X\to X$ be a $C^1$ map with all its periodic points hyperbolic, 
and 
$a_{n_1},\ldots, a_{n_l}$, are the basic eigenvalues of $f_{*}$. 

If $a_{n_i}\neq 1$ for all $n_i$ odd and there exists $j\in\{1,\ldots,l\}$
such that $|a_{n_j}|>1$ then $f$ has infinitely many periodic points.
\end{theorem}
\begin{proof}
Due to~(\ref{eqn:L-m}) the sequence of the Lefschetz numbers 
$\{L(f^m)\}_m$ is unbounded if and only if
$a_{n_i}\neq 1$ for all $n_i$ odd and there exists $j\in\{1,\ldots,l\}$
such that $|a_{n_j}|>1$.

According to~\cite[Theorem 2.2]{babo}  the sequence 
$\{L(f^m)\}_m$ is unbounded is equivalent to the fact there exist a zero or  pole of $\zf(t)$ with modulus less than $1$.
Since the expression of the Lefschetz zeta function given by~(\ref{eqn:lzf-product}) and the numbers $a_i$ are integers, hence
there exits at least one term of the form $(1-ct)$, with $|c|>1$, which is not cancel out in the formula~(\ref{eqn:lzf-product}).
Therefore $\zf(t)$ cannot be written in the form~(\ref{eqn:zeta-franks}) 
so by Theorem~\ref{thm:franks}, the map $f$ has infinitely many periodic points.
\end{proof}

Let us consider the following example: $f:X\ra X$, where
$X=\SSS^1\times\SSS^2\times\SSS^3$, the induced maps on homology
are: $f_{*1}=(a_1)$, $f_{*2}=(a_2)$, $f_{*4}=(a_1a_3)$, $f_{*5}=(a_2a_3)$, $f_{*6}=(a_1a_2a_3)$ and
$$
f_{*3}=\left(\begin{array}{cc}
                 a_3 & c \\
                 0 & a_1a_2
                 \end{array}\right),
$$
where $a_1,a_2,a_3\in\Z$ are the basic eigenvalues of $f_{*}$ and $c$ an integer. %for more details see~\cite[Example 2]{Si:2020}.
According~(\ref{eqn:lzf-product}), its Lefschetz zeta function is
$$
\zf(t)=\dfrac{(1-a_1t)(1-a_3t)(1-a_1a_2t)(1-a_2a_3t)}{(1-t)(1-a_2t)(1-a_1a_3t)(1-a_1a_2a_3t)}.
$$
Note that this expression is of type~(\ref{eqn:zeta-franks}) if and only if
$a_1=1$, $a_3=1$ (in any of these cases $\zf(t)=1$), or $|a_i|\leq 1$ for $1\leq i\leq 3$.

%%%%%%%%%%%%%%%%%%%%%%%%%%%%%%%%%%%%

\section{Transversal maps}\label{s:transversal}

A  {\em transversal map} $f$ on a compact differentiable manifold $X$
is a $C^1$ map $f:X\to X$, such that $f(X)\subset\rm{Int}(X)$ and for
every positive integer  $m$  at each point $x$ fixed by $f^m$ we have
that $1$ is not an eigenvalue of $Df^m(x)$, i.e.
$\det(Id-Df^m(x))\neq 0$. 

We denote the set of periods of a map $f$, by $\mbox{Per}(f)$.

One of the most important results concerning the periodic structure of the transversal maps on any manifold, i.e. the description of its set of periods, is the following theorem:

\begin{theorem}[\cite{Ll:1993, GJLlT}]\label{thm:transversal}
Let $X$ be a compact manifold and $f:X\to X$ be a transversal map.
Suppose $\ell(f^m)\neq 0$, for some $m$. Then
\begin{itemize}
\item[(a)] If $m$ is odd, then $m\in\mbox{Per}(f)$.
\item[(b)] If $m$ is even, then $m$ or $m/2$ is in $\mbox{Per}(f)$.
\end{itemize}
\end{theorem}

In~\cite{G-LL:2013} this theorem was used  for transversal self-maps on the product of two spheres. And in~\cite{LL-S:2019} it was also used for  transversal maps on a product of any numbers of spheres such that all the homology groups are trivial or one-dimensional, the so-called sum-free product.
Lemma~\ref{lemma:transversal} generalizes Theorems 1.2 and 1.4 of 
~\cite{LL-S:2019}. And Theorem~\ref{thm:main-transversal} is a more general version of Corollary 1.3 and 1.5, moreover its proof is simpler.

\begin{lemma}\label{lemma:transversal}
Let $X=\SSS^{n_1}\times\cdots\times\SSS^{n_l}$, with  $1\leq n_1<\cdots
<n_l$.  Let $f$ be a transversal self--map
on $X$,  with basic eigenvalues of $f_*$: $a_{n_1},\ldots,a_{n_l}$.
Assume that the basic eigenvalues are not zero and there is no $n_i$ odd with $a_{n_i}=1$.
%\textcolor{red}{that $(-1)^{n_i}a_i\neq -1$, for all $1\leq i\leq l$. ()}
\begin{enumerate}
\item[(a)] If  $|a_j|>1$ for $1\leq j\leq l$, then there exists $N>0$, such that $\ell(f^m)\neq 0$, for all $m>N$.
\item[(b)] If there is $n_i$ even with $a_i=-1$ and there is $|a_j|>1$ for $j\in\{1,\ldots,l\}$, then $\ell(f^{2m+1})=0$ for all $m>1$ and there exists $N>0$, such that $\ell(f^{2m})\neq 0$, for all $m>N$.
\item[(c)] If there is $n_i$ odd with $a_i=-1$ and there is $|a_j|>1$ for $j\in\{1,\ldots,l\}$, then $\ell(f^{2m})=0$ for all $m>1$ and there exists $N>0$, such that $\ell(f^{2m+1})\neq 0$, for all $m>N$.
\item[(d)]
Let ${\mathcal I}:=\{j\in\{1,\ldots,l\}\, :\, |a_j|>1\}$. 
If ${\mathcal I}\neq\emptyset$ and $(-1)^{n_i}a_i=(-1)^{n_i}a_i^2=1$, for 
$i\notin{\mathcal I}$, then 
 then  there exists $N>0$, such that $\ell(f^{m})\neq 0$, for all $m>N$.
\item[(e)] Suppose $|a_i|=1$ for all $1\leq i\leq l$. %then $\ell(f^m)=0$ for $m>1$.
\begin{enumerate}
\item[(e-1)] If $a_{n_i}=1$ for all $1\leq i\leq l$ then $\ell(f^m)=0$, for all $m>1$.
\item[(e-2)] If there are no $n_i$ even with $a_{n_i}=-1$ and there is $n_j$ odd with $a_{n_j}=-1$ then $\ell(f^m)\neq  0$,  when  $m=2^s$, for some $s$ and $\ell(f^m)=0$ otherwise.
\item[(e-3)] If all $n_i$ are even and there exists $a_{n_j}=-1$, for some $1\leq j\leq l$ then $\ell(f^m)\neq 0$ when  $m=2^s$, for some $s$ and $\ell(f^m)=0$ otherwise.
\end{enumerate}
\end{enumerate}

\end{lemma}

From Lemma~\ref{lemma:transversal} and Theorem~\ref{thm:transversal} we get the
following result:

\begin{theorem}\label{thm:main-transversal}
Let $X=\SSS^{n_1}\times\cdots\times\SSS^{n_l}$, with  $1\leq  n_1<\cdots
<n_l$.  Let $f$ be a transversal self--map
on $X$,  with basic eigenvalues of $f_*$: $a_{n_1},\ldots,a_{n_l}$ and all of them are different from zero.
Assume there is no $n_i$ odd with $a_i=1$. 
%\textcolor{red}{that $(-1)^{n_i}a_i\neq -1$, for all $1\leq i\leq l$. }
\begin{enumerate}
\item  If  $|a_j|>1$ for $1\leq j\leq l$, then there exists $N>0$, such that
\begin{itemize}
\item the number $m$ is in $\mbox{Per}(f)$, if $m>N$ is odd, and 
\item $m$ or $m/2$ is in $\mbox{Per}(f)$, if $m>N$ is even.
\end{itemize}
\item If there is $n_i$ even with $a_i=-1$ and there is $|a_j|>1$ for $j\in\{1,\ldots,l\}$, then there exists $N>0$ such that $m$ or $m/2$ is in $\mbox{Per}(f)$, if $m>N$ is even.
\item If there is $n_i$ odd with $a_i=-1$ and there is $|a_j|>1$ for $j\in\{1,\ldots,l\}$, then  there exists $N>0$ such that  $m\in\mbox{Per}(f)$, if $m>N$ is odd.
\item If there are no $n_i$ even with $a_{n_i}=-1$ and there is $n_j$ odd with $a_{n_j}=-1$  then $m$ or $m/2$  is in $\mbox{Per}(f)$, when $m$ is a power of $2$.
\item If all $n_i$ are even and  $|a_{n_i}|=1$, and   there exists $a_{n_j}=-1$, for some $1\leq j\leq l$, then $m$ or $m/2$  is in $\mbox{Per}(f)$, when $m$ is a power of $2$.
\end{enumerate}

\end{theorem}

\begin{proof}[Proof of Lemma~\ref{lemma:transversal}]
Using the definition of Lefschetz numbers of period $m$, Proposition~\ref{prop:L-numbers} and the properties of the M\"obius function, we have for $m>1$:
\begin{eqnarray*}
\ell(f^m)&=&  \sum_{r|m} \mu(r) L(f^{m/r})\\
            &=&  \sum_{r|m} \mu(r) \left(\prod_{i=i}^l (1+(-1)^{n_i} a_{n_i}^{m/r})\right)\\
            &=&  \prod_{i=i}^l \left(\sum_{r|m} \mu(r) (1+(-1)^{n_i} a_{n_i}^{m/r})\right)\\
           &=&\prod_{i=i}^l \left(\sum_{r|m} \mu(r)\right)+\prod_{i=i}^l (-1)^{n_i}\left(\sum_{r|m} \mu(r)  a_{n_i}^{m/r}\right).
         %  &=& \prod_{i=1}^l (-1)^{n_i}Q_m(a_{n_i}).
\end{eqnarray*}
If we set $Q_m(x):= \sum_{r|m}\mu(r) x^{m/r}$, the previous computations yields:
\begin{equation}\label{eqn:lQ}
\ell(f^m)= \prod_{i=1}^l (-1)^{n_i}Q_m(a_{n_i}).
\end{equation}

Note that if for some $i\in\{1,\ldots, l\}$  $a_{n_i}=1$ and $n_i$ odd   then
$L(f^m)=0$, hence $\ell(f^m)=0$, for  all $m$. 

In order to show that $\ell(f^m)$ are bounded away of zero, we need to compute the growth of the family of polynomials $Q_m(x)$. This is done in the following proposition:
  
\begin{proposition}\label{prop:polyQ}
If $|a|>1$ then there exists a constant $C_a>0$ (which only depends of $a$) and a positive integer $N$ such  that $|Q_m(a)|\geq C_a |a|^m$ for $m>N$.
\end{proposition}

Note that if $|a_i|>1$ for all $1\leq i\leq l$, the equality~(\ref{eqn:lQ}) and Proposition~\ref{prop:polyQ} yields:
there exist positive constants $C_1,\ldots C_l$ and  $N$ such that
$$
|\ell(f^m)|=\prod_{i=1}^l |Q_m(a_{n_i})|\geq C_1\cdots C_l |a_{n_1}\cdots a_{n_l}|^m,
$$
for $m>N$, hence $\ell(f^m)\neq 0$, for $m>N$.
This proves statement (a).

\medskip

If there is $n_i$ even such that $a_i=-1$ then $L(f^m)=0$ for $m$ odd, therefore $\ell(f^m)=0$ for all $m$ odd.
By Proposition~\ref{prop:polyQ}, if there is $j\in\{1,\ldots,l\}$ with $|a_j|>1$ then there  exists $N$ such that $|Q_m(a_{n_j})|>C_j|a_{n_j}|^m$,
for $m>N$ and $C_j>0$. Hence
$|\ell(f^m)|\geq C_j |a_{n_j}|^m>0$, for $m>N$ and even.
This completes the proof of statement (b).

\medskip

A similar analysis is done when there is $n_i$ odd with $a_i= -1$ and there exists $|a_j|>1$. In this case $\ell(f^m)=0$ for all $m$ even and 
$|\ell(f^m)|\geq C_j |a_{n_j}|^m>0$, for large and even  $m$.
This proves statement (c).

\medskip

Note that in the case of (d), we have $L(f^m)\neq 0$ for all  $m$. 
By~(\ref{eqn:lQ}) and  Proposition~\ref{prop:polyQ} we have
$$
|\ell(f^m)|=\prod_{i=1}^l |Q_m(a_{n_i})| =\left(\prod_{i\in{\mathcal I}} 2\right)
\prod_{i\notin{\mathcal I}} |Q_m(a_{n_i})|
\geq 2^{\#{\mathcal I}} \prod_{i\notin{\mathcal I}} C_i |a_{n_i}|^m.
$$
for sufficient large $m$.
This completes the proof of statement  (d).

\medskip

If $|a_i|=1$ for all $1\leq i\leq l$ then~(\ref{eqn:L-m}) can be written as
$$
L(f^m)=\left(\prod_{\substack{n_i \text{ odd}\\ a_{n_i}=-1}}(1+(-1)^{n_i}a_{n_i}^m)\right)\left(\prod_{\substack{n_i \text{ even}\\ a_{n_i}=1}}(1+(-1)^{n_i}a_{n_i}^m)\right)\left(\prod_{\substack{n_i \text{ even}\\ a_{n_i}=-1}}(1+(-1)^{n_i}a_{n_i}^m)\right).
$$

For simplicity we introduce the sets:
$$
\begin{array}{ccc}
O^-:=\{i\,:\, n_i \text{ odd }, a_{n_i}=-1\}, &
E^+:=\{i\,:\, n_i \text{ even }, a_{n_i}=1\} &
E^-:=\{i\,:\, n_i \text{ even }, a_{n_i}=-1\}.
\end{array}$$

Note that if $O^-=E^-=\emptyset$ then $L(f^m)=2^l$ then $\ell(f^m)=0$ for $m>1$, and $\ell(f)=1$. This prove statement (e-1).

If $O^-\neq \emptyset$ and $E^-=\emptyset$ then
$$
L(f^m)=\prod_{n_i\in O^-}(1-(-1)^m)\prod_{n_i\in E^-} 2=
\left\{\begin{array}{ll} 2^{\# O^+}2^{\# E^-}  & m \text{ odd}\\
0 & m\text{ even}.
\end{array}\right.
$$
In the case of
$m$  odd it follows
$$
\ell(f^m)=\sum_{r|m} \mu(r)L(f^{m/r})=2^{\# O^+}2^{\# E^-}\sum_{r|m} \mu(r).
$$
Therefore $\ell(f^m)=0$ for $m>1$ and $\ell(f)\neq 0$.

If $m$ is even then
\begin{eqnarray*}
 %$$
 \ell(f^m)&=&\sum_{\substack{r|m\\ r \text{ even}}}\mu(m/r) L(f^r)+\sum_{\substack{r|m\\ r \text{ odd}}}\mu(m/r) L(f^r)\\
 &=&\sum_{\substack{r|m\\ r \text{ odd}}}\mu(m/r) L(f^r)
 = 2^{\# O^+}2^{\# E^-} \left(\sum_{\substack{r|m\\ r \text{ odd}}}\mu(m/r)\right).
 %$$
 \end{eqnarray*}
It can be easily check that
 $$
 \sum_{\substack{r|m\\ r \text{ odd}}}\mu(m/r)=
 \left\{\begin{array}{ll}
  1 & \mbox{ if } m=2^s \mbox{ for some } s>0,\\
  0 & \text{otherwise}.
\end{array}  
 \right.
 $$
This proves statement (e-2).
 
 \medskip
 
 If all $n_i$ are even, i.e. $O^-=\emptyset$, and $E^-\neq \emptyset $ then %$L(f^m)=0$, for $m$ odd.
% In the case of $m$ even we get
 $$
 L(f^m)=\prod_{n_i\in E^+} 2 \prod_{n_i\in E^-} (1+(-1)^m)=
 \left\{\begin{array}{ll} 2^{\# E^+}2^{\# E^-}  & m \text{ even}\\
0 & m\text{ odd}.
\end{array}\right.
 $$
 Therefore
 $\ell(f^m)=0$, for $m$ odd.
 In the case of $m$ even
 \begin{eqnarray*}
 %$$
 \ell(f^m)&=&\sum_{\substack{r|m\\ r \text{ even}}}\mu(m/r) L(f^r)+\sum_{\substack{r|m\\ r \text{ odd}}}\mu(m/r) L(f^r)\\
 &=&\sum_{\substack{r|m\\ r \text{ even}}}\mu(m/r) L(f^r)
 = 2^{\# E^+}2^{\# E^-} \left(\sum_{\substack{r|m\\ r \text{ even}}}\mu(m/r)\right).
 %$$
 \end{eqnarray*}
 It can be easily check that
 $$
 \sum_{\substack{r|m\\ r \text{ even}}}\mu(m/r)=
 \left\{\begin{array}{ll}
  -1 & \mbox{ if } m=2^s \mbox{ for some } s>0,\\
  0 & \text{otherwise}.
\end{array}  
 \right.
 $$
 Therefore $\ell(f^m)\neq 0$ if $m$ is a power of $2$, and  otherwise $\ell(f^m)=0$. 
 This completes the proof of statement (e-3).
%$ L(f^m)=0$ or $L(f^m)=2^J$, for all $m$, where $J=\#\{i\, :\, (-1)^{n_i}a_i=1\}$, in former case obviously
%\comment{Study this case carefully}
\end{proof}

\begin{proof}[Proof of Proposition~\ref{prop:polyQ}]
The proof of this proposition is quite standard, we present it here for completeness sake.

Using the triangular inequality and the definition of $Q_m(x)$ we have
\begin{eqnarray*}
|Q_m(a)|&\geq& |a|^m-|\sum_{\substack{r|m\\r\neq 1}} \mu(r) a^{m/r}|
\geq |a|^m |1-\sum_{\substack{r|m\\r\neq 1}} \mu(r) a^{m/r-m}|\\
&\geq& |a|^m\left( 1-\sum_{\substack{r|m\\r\neq 1}} |a^{m/r-m}|\right)
\geq |a|^m \left( 1- |a^{m'-m}|\sum_{r|m}1\right),
\end{eqnarray*}
where $m':=\max\{m/r\, :\, r|m, \, r\neq 1\}$. If $m$ is prime then  $m'=1$, and if $m$ is a composite number then $m'\geq \sqrt{m}$.
Therefore
$$
|Q_m(a)|\geq|a|^m \left( 1-\sigma(m))|a^{m'-m}|\right)\geq |a|^m 
\left( 1-\dfrac{\sigma(m)}{|a^{m-\sqrt{m}}|}\right)
$$
where $\sigma(m)$ is the number of divisors of m. 
For large $m$, $\sigma(m)\leq C' m^\epsilon$, with $\epsilon>0$ arbitrary small and $C'$ a positive constant~(\emph{cf.}~\cite[Theorem 315, pp. 260]{hardy-wright}.
Therefore $m^{\epsilon}a^{\sqrt{m}-m}$ is arbitrary small for large $m$.
%t is known that for large m, ?(m) ? C?m?, where C? is a positive constant and ? is arbitrary small, for more details see [10, Theorem 315, pp. 260]
Hence there exists $N>0$, such that for $m\geq N$ and composite it follows
$$
|Q_m(a)|\geq   |a|^m\left(1- \dfrac{C' m^\epsilon}{ |a^{m-\sqrt{m}}|}\right)\geq C |a|^m
$$
for some constant $C>0$, which only depends of $a$.
If $m$ is a large prime, we get 
$$ 
|Q_m(a)|\geq  |a|^m 
\left( 1-\dfrac{\sigma(m)}{|a^{m-1}|}\right),
$$
it also yields $|Q_m(a)|\geq  C |a|^m$, for a positive constant $C$. 
\end{proof}

%%%%%%%%%%%%%%%%%%%%%%%%%%%%%%%%%%%%

\section{Remarks, applications and open questions}\label{s:remarks}
\begin{enumerate}

\item 
An interesting  application of the results of this article is the study and characterization of the periodic structure of $C^1$ self-maps on Lie groups, since the homology of Lie groups satisfying some topological properties are the same of the product of spheres of different odd dimensions, due to Hopf's Theorem~(\emph{cf.}~\cite{kumpel}):

\begin{theorem}[Hopf's Theorem]\label{thm:hopf}
Let $X=G$ be a compact, connected, simply connected, simple Lie group,
 then the homology groups of $G$ over the reals are the same as  the homology groups of $\SSS^{n_1}\times\cdots\times\SSS^{n_l}$, where $n_1,\ldots,n_l$ are odd, $l$ is the rank of $G$ and $n_1+\cdots+n_l$ is the dimension of $G$.\end{theorem}

Therefore the results of this article (Theorems~\ref{thm:mperl}, \ref{thm:pp} and~\ref{thm:main-transversal}) and Hopf's Theorem   yield the following corollaries:

\begin{corollary}\label{coro:ms-lie}
Let $G$ be a Lie group satisfying the hypotheses of Theorem~\ref{thm:hopf} and
$f:G\to G$ be a Morse-Smale diffeomorphisms such that 
$a_{n_1},\ldots, a_{n_l}$, are the basic eigenvalues of $f_{*}$. Then the minimal set of Lefschetz periods is either the empty set or $\{1\}$.
Moreover
\begin{enumerate}
\item If $a_{n_i}=-1$ for all $1\leq i\leq l$ then $\mperL=\{1\}$.
\item If $a_{n_i}=1$ for some $1\leq i\leq l$ then $\mperL=\emptyset$.
\end{enumerate}
\end{corollary}

\begin{corollary}\label{coro:pp-lie}
Let $G$ be a Lie group satisfying the hypotheses of Theorem~\ref{thm:hopf} and
$f:G\to G$ be a $C^1$ map with all its periodic points hyperbolic such that 
$a_{n_1},\ldots, a_{n_l}$, are the basic eigenvalues of $f_{*}$.
If $a_{n_i}\neq 1$ for all $n_i$  and there exists $j\in\{1,\ldots,l\}$
such that $|a_{n_j}|>1$ then $f$ has infinitely many periodic points.
\end{corollary}

\begin{corollary}\label{coro:transversal-lie}
Let $G$ be a Lie group satisfying the hypotheses of Theorem~\ref{thm:hopf} and
$f:G\to G$ be a transversal self--map on $G$
  Let $f$ be a transversal self--map
on $G$,  with basic eigenvalues of $f_*$: $a_{n_1},\ldots,a_{n_l}$.
Assume that $a_i\neq 1$.
\begin{enumerate}
\item  If  $|a_{n_j}|>1$ for $1\leq j\leq l$, then there exists $N>0$, such that
\begin{itemize}
\item the number $m$ is in $\mbox{Per}(f)$, if $m>N$ is odd, and 
\item $m$ or $m/2$ is in $\mbox{Per}(f)$, if $m>N$ is even.
\end{itemize}
\item If $a_{n_i}=-1$ and there is $|a_{n_j}|>1$ for $j\in\{1,\ldots,l\}$, then  there exists $N>0$ such that  $m\in\mbox{Per}(f)$, if $m>N$ is odd.
\item If $|a_{n_i}|=1$ for all $1\leq i\leq l$ and there is $a_{n_j}=-1$ for some $n_j$ then $\{m,m/2\}\subset\mbox{Per}(f)$ when $m$ is a power of $2$.
\end{enumerate}
\end{corollary}

Among the classical Lie groups that satisfy the hypotheses of the Hopf's Theorem are: $SU(n)$, $SO(n)$ and $Sp(n)$.

\item
A general goal of this theory is to extend the results of
 of this article to self-maps on the space
$X=X(n_1,s_1)\times\cdots\times X(n_l,s_l)$, where $X(n_i,s_i):=\underbrace{\SSS^{n_i}\times\cdots\times\SSS^{n_i}}_{s_i}$, with $1\leq n_1<\cdots<n_l$, and $s_i\geq 1$.
It would quite interesting to achieve this goal
We recall that in the particular case of $X=X(n,s)$ for some $n\geq 1$ and $s>1$, the study of the minimal sets Lefschetz periods and the computation of the Lefschetz zeta function for quasi-unipotent maps was done in~\cite{BGS1}.
In the particular situation of the torus, i.e. $n=1$,  the study was done previously  in~\cite{BGMS,G-LL1}.
Regarding the periodic structure of transversal maps for $X=X(n,s)$, for $n\geq 1$ and  $s> 1$, see~\cite{Si:2020-2}.
However we believe  new techniques is required to achieve these goals.
\item
The precise values of the exponent $\alpha$ and $\beta$ in~(\ref{eqn:lzf-qu}) can be computed as follows:
Let 
\begin{eqnarray*}
e(k)&: = &\#\{(j_1,\ldots, j_r)\, :\, j_1+\cdots+j_r=k \text{ and } a_{j_1}\cdots a_{j_r}=1\}
\text { and } \\
o(k)&:=&\#\{(j_1,\ldots, j_r)\, :\, j_1+\cdots+j_r=k \text{ and } a_{j_1}\cdots a_{j_r}=-1\},
\end{eqnarray*}
 where the symbol $\#$ means the cardinality of the set.
From the identity~(\ref{eqn:lzf-prod1}), it follows $\alpha=\sum_{k=1}^N e(k)(-1)^{k+1}-1$
and $\beta=\sum_{k=1}^N e(k)(-1)^{k+1}$.
%\comment{Check}

\item The results of this article can be generalized in the context of some rational exterior spaces. For the definition and basic properties of self-maps on these spaces, see~\cite{haibao,graff:2000}. 
The dynamics of the transversal maps on rational exterior spaces of a given fixed rank, was studied in~\cite{Si:2020-2}.
However, we do not present the generalization of our results in this context since it is out the scope of the present article.
\end{enumerate}

\noindent
\textbf{Availability of data and material:} Data sharing not applicable to this article as no datasets were generated or analyzed during the current study.

\end{document}